\documentclass{amsart}
\usepackage{graphicx}
\usepackage{amssymb}
\usepackage{amsfonts}
\swapnumbers
\newtheorem{theorem}{Theorem}[section]
\newtheorem{lemma}[theorem]{Lemma}
\newtheorem{corollary}[theorem]{Corollary}
\newtheorem{proposition}[theorem]{Proposition}
\theoremstyle{definition}

\newtheorem{assumption}[theorem]{Assumption}
\newtheorem{remark}[theorem]{Remark}

\numberwithin{equation}{section}
 \theoremstyle{plain}    
 
 \numberwithin{equation}{section} %% Comment out for sequentially-numbered
 \numberwithin{figure}{section} %% Comment out for sequentially-numbered
 \theoremstyle{plain}    
 \theoremstyle{plain}    
 \theoremstyle{remark}    
 \newtheorem*{acknowledgement*}{Acknowledgement} 
\newcommand{\cF}{{\mathcal F}}
\newcommand{\cG}{{\mathcal G}}
\newcommand{\cH}{{\mathcal H}}

\newcommand{\te}{{\theta}}

\newcommand{\Om}{{\Omega}}
\newcommand{\om}{{\omega}}
\newcommand{\ve}{{\varepsilon}}
\newcommand{\del}{{\delta}}

\newcommand{\gam}{{\gamma}}

\newcommand{\vp}{{\varpi}}
\newcommand{\io}{{\iota}}
\newcommand{\up}{{\upsilon}}
\newcommand{\Up}{{\Upsilon}}
\newcommand{\Sig}{{\Sigma}}
\newcommand{\sig}{{\sigma}}
\newcommand{\al}{{\alpha}}
\newcommand{\be}{{\beta}}
\newcommand{\ka}{{\kappa}}

\newcommand{\bbC}{{\mathbb C}}

\newcommand{\bbN}{{\mathbb N}}

\newcommand{\bbR}{{\mathbb R}}

\newcommand{\bbI}{{\mathbb I}}
\begin{document}
\title[]{Strong approximations for nonconventional sums\\ with applications
to law of iterated logarithm\\ and almost sure central limit theorem}%
 \vskip 0.1cm 
 \author{ Yuri Kifer\\
 \vskip 0.1cm
Institute of Mathematics\\
The Hebrew University of Jerusalem}% 
\email{kifer\@@math.huji.ac.il}
\address{Institute of Mathematics, Hebrew University, Jerusalem 91904,\linebreak
 Israel}

\thanks{ }
\dedicatory{  }
 \date{\today}
\begin{abstract}\noindent
We improve, first, a strong invariance principle from \cite{Ki2} for 
nonconventional sums of the form 
$\sum_{n=1}^{[Nt]}F\big(X(n),X(2n),...,X(\ell n)\big)$
(normalized by $1/\sqrt N$) where $X(n),n\geq 0$'s is a sufficiently fast 
mixing vector process with some moment conditions and stationarity properties
and $F$ satisfies some regularity conditions. Applying this result we obtain 
next a version of the law of iterated logarithm for such sums, as well as an
almost sure central limit theorem. Among motivations for such results are their
applications to multiple recurrence for stochastic processes and dynamical 
systems.
\end{abstract}
%\footnotetext[1]{} 
\maketitle
\markboth{Yu.Kifer}{Strong approximations} 
\renewcommand{\theequation}{\arabic{section}.\arabic{equation}}
\pagenumbering{arabic}

\section{Introduction}\label{sec1}\setcounter{equation}{0}

In this paper we study almost sure limit theorems for nonconventional sums
 of the form
 \begin{equation}\label{1.1}
\Xi(t)=\sum_{1\leq n\leq t}\big(F\big(X(n),X(2n),..., X(\ell n)\big)
-\bar F\big)
\end{equation}
where $X(n),n\geq 0$'s is a sufficiently fast mixing 
vector valued process with some moment conditions and stationarity properties,
$F$ is a continuous function with polinomial growth and certain regularity
properties, 
$\bar F=\int Fd(\mu\times\cdots\times\mu)$ and $\mu$ is the distribution
of $X(0)$. The name "nonconventional" comes from \cite{Fu} where ergodic
theorems for averages $N^{-1}\Xi(N)$ were studied in the case when $X(n)=
X(n,x)=T^nx$ with $T$ being a measure preserving ergodic transformation.
We observe that the topic of nonconvenional ergodic theorems was extensively
studied during the last 30 years.

Recently the setup of nonconventional sums was studied from the probabilistic 
point of view and the strong law of large numbers, the functional central
limit theorem and a version of the strong invariance principle (called also
a strong approximation theorem) were obtained in \cite{Ki1}, \cite{KV} and
\cite{Ki2}, respectively. In this paper we partially improve the strong
invariance principle from \cite{Ki2} which enable us both to obtain a better
than in \cite{Ki2} version of the law of iterated logarithm and, moreover, to
derive an almost sure central limit theorem for sums $\Xi(t)$. 
We will show in this paper that the sum $\Xi(t)$ can be approximated as
$t\to\infty$ with an error term of order $t^{\frac 12-\gam},\,\gam>0$ by
certain Gaussian process $G(t)$ having, in general, dependent increments.
This will enable us to obtain the law of iterated logarithm and the
almost sure central limit by certain modifications of familiar proofs.
We observe that in \cite{Ki2} strong approximations were obtained only
for certain components of the sum $\Xi(t)$ which did not allow to 
derive the corresponding result for the whole sum. On the other
hand, terms $X(q(n))$ with nonlinear $q(n)$ were considered in \cite{Ki2}
while we do not deal with them here.

One of motivations for nonconventional limit theorems comes from multiple
recurrence problems. Let, for instance, $F(x_1,...,x_\ell)=x_1\cdots x_\ell$
and $X(n)=\bbI_A(\xi_n)$ where $\bbI_A$ is the indicator of a set $A$ and 
$\xi_n$ is either a dynamical system $\xi_n=\xi_n(\om)=T^n\om$ or a Markov
chain. Then $\Xi(N)=\sum_{n=1}^NF(X(n),...,X(\ell n))$ is the number of
events $\{\xi_n\in A,\,\xi_{2n}\in A,...,\xi_{\ell n}\in A\}$ for $n\leq N$, 
and so
our results describe limiting behavior of such quantities. This also can be
described as a number of arithmetic progressions of length $\ell$ starting
at 0 whose difference is between 1 and $N$ and such that at any positive 
time belonging to such progression the process $\xi$ is contained in $A$.

As in \cite{Ki1}, \cite{Ki2} and \cite{KV} our results hold true when, for
instance, $X(n)=T^nf$ where $f=(f_1,...,f_{\wp})$, $T$ is a mixing subshift of
finite type, a hyperbolic diffeomorphism (see \cite{Bo}) or an expanding
transformation taken with a Gibbs invariant measure, as well, as in the case
when $X(n)=f(\xi_k)$, $f=(f_1,...,f_\wp)$ where $\xi_n$ is a Markov chain
satisfying the Doeblin condition (see, for instance, \cite{IL}) considered
as a stationary process with 
respect to its invariant measure. Furthermore, our results are applicable 
to other dynamical systems such as the Gauss map of the interval (see, for
instance, \cite{IL} or \cite{Hei}) and a large class of transformations having
a spectral gap of their transfer operator which ensure their fast mixing 
properties. On the probabilistic side our results work for Markov processes
having transition operators with a spectral gap, in particular, for
Ornstein-Uhlenbeck type processes.

\section{Preliminaries and main results}\label{sec2}\setcounter{equation}{0}

As in \cite{Ki2} we deal with the setup which consists of a  $\wp$-dimensional 
stochastic process
$\{X(n),  n=0,1,...\}$ on a probability space $(\Om,\cF,P)$ and of a family
of $\sig$-algebras \hbox{$\cF_{kl}\subset\cF,\, -\infty\leq k\leq
l\leq\infty$} such that $\cF_{kl}\subset\cF_{k'l'}$ if $k'\leq k$ and
$l'\geq l$. The dependence between two sub $\sig$-algebras $\cG,\cH\subset\cF$
is measured often via the quantities
\begin{equation}\label{2.1}
\varpi_{q,p}(\cG,\cH)=\sup\{\| E\big [g|\cG\big]-E[g]\|_p:\, g\,\,\mbox{is}\,\,
\cH-\mbox{measurable and}\,\,\| g\|_q\leq 1\},
\end{equation}
where the supremum is taken over real functions and $\|\cdot\|_r$ is the
$L^r(\Om,\cF,P)$-norm. Then more familiar $\al,\rho,\phi$ and $\psi$-mixing 
(dependence) coefficients can be expressed in the form (see \cite{Bra}, Ch. 4 ),
\begin{eqnarray*}
&\al(\cG,\cH)=\frac 14\varpi_{\infty,1}(\cG,\cH),\,\,\rho(\cG,\cH)=\varpi_{2,2}
(\cG,\cH)\\
&\phi(\cG,\cH)=\frac 12\varpi_{\infty,\infty}(\cG,\cH)\,\,\mbox{and}\,\,
\psi(\cG,\cH)=\varpi_{1,\infty}(\cG,\cH).
\end{eqnarray*}
The relevant quantities in our setup are
\begin{equation}\label{2.2}
\varpi_{q,p}(n)=\sup_{k\geq 0}\varpi_{q,p}(\cF_{-\infty,k},\cF_{k+n,\infty})
\end{equation}
and accordingly
\[
\al(n)=\frac{1}{4}\varpi_{\infty,1}(n),\,\rho(n)=\varpi_{2,2}(n),\,
\phi(n)=\frac 12\varpi_{\infty,\infty}(n)\,\,\mbox{and}\,\,\ \psi(n)=
\varpi_{1,\infty}(n).
\]
Our assumptions will require certain speed of decay as $n\to\infty$ of both
the mixing rates $\varpi_{q,p}(n)$ and the approximation rates defined by
\begin{equation}\label{2.3}
\beta_p (n)=\sup_{m\geq 0}\|X(m)-E\big (X(m)|\cF_{m-n,m+n}\big)\|_p.
\end{equation}
In what follows we can always extend the definitions of $\cF_{kl}$
given only for $k,l\geq 0$ to negative $k$ by defining  $\cF_{kl}=\cF_{0l}$ 
for $k<0$ and $l\geq 0$. Furthermore, we do not require stationarity of the
process $X(n), n\geq 0$ assuming only that the distribution of $X(n)$ does not
depend on $n$ and the joint distribution of $\{X(n), X(n')\}$  depends only on
$n-n'$ which we write for further references by 
\begin{equation}\label{2.4}
X(n)\stackrel {d}{\sim}\mu\,\,\mbox{and}\,\,
(X(n),X(n'))\stackrel {d}{\sim}\mu_{n-n'}\,\,\mbox{for all}\,\,n,n'
\end{equation}
where $Y\stackrel {d}{\sim}Z$ means that $Y$ and $Z$ have the same distribution.

Next, let $F= F(x_1,...,x_\ell),\, x_j\in\bbR^{\wp}$ be a function on 
$\bbR^{\wp\ell}$ such that for some $\iota,K>0,\ka\in (0,1]$ and all 
$x_i,y_i\in\bbR^{\wp}, i=1,...,\ell$, 
\begin{equation}\label{2.5}
|F(x_1,...,x_\ell)-F(y_1,...,y_\ell)|\leq K\big(1+\sum^\ell_{j=1}|x_j|^\iota+
\sum^\ell_{j=1} |y_j|^\iota \big)\sum^\ell_{j=1}|x_j-y_j|^\ka
\end{equation}
and 
\begin{equation}\label{2.6} 
|F(x_1,...,x_\ell)|\leq K\big( 1+\sum^\ell_{j=1}|x_j|^{\iota} \big).
\end{equation}
The above assumptions allow us to consider, for instance, functions 
$F$ polynomially dependent on their arguments, in particular, the product
function $F(x_1,...,x_\ell)=x_1,...,x_\ell$.
To simplify formulas we assume a centering condition
\begin{equation}\label{2.7}
\bar F=\int F(x_1,...,x_\ell)\,d\mu(x_1)\cdots d\mu(x_\ell)=0
\end{equation}
which is not really a restriction since we always can replace $F$ by 
$F-\bar F$. 

For each $\theta>0$ set
\begin{equation*}
\gamma_\theta^\theta = \|X\|_\theta^\theta= E|X(n)|^\theta  =
\int \|x\|^\theta d\mu .
\end{equation*}
Our main result relies on

\begin{assumption}\label{ass2.1} With $d=(\ell-1)\wp$ there exist $p,q\ge 1$
 and $\delta,m >0$ with   $ \delta<\ka-\frac dp$ satisfying 
\begin{equation*}
\sum_{n=0}^\infty n\varpi_{q,p}(n)<\infty,
\end{equation*}
\begin{equation*}
\sum_{r=0}^\infty r^{\frac {16}\del}\beta^\delta_q(r)<\infty,
\end{equation*}
\begin{equation*}
\gamma_{m}<\infty\,\,,\gam_{2q(\io+2)}<\infty\,\,\mbox{with}\,\,
 \frac{1}{2+\del}\ge \frac{1}{p}+\frac{\iota+2}{m}+\frac{\delta}{q}.
 \end{equation*}
 \end{assumption}
 
 A reader willing to avoid some visual technicalities may be advised to
 consider bounded Lipschitz continuous functions $F$ and exponentially fast
 decaying $\varpi_{q,p}(n)$ $\be_q(n)$ as $n\to\infty$ but since we mostly
 rely in this paper on estimates from \cite{KV} and \cite{Ki2} this would
 not matter much here.
 
 As in \cite{KV} a crucial part of our approach is the representation
 of $F=F(x_1,...,x_\ell)$ in the form
\begin{equation}\label{2.8}
F=F_1(x_1)+\cdots+F_\ell(x_1, x_2,\ldots, x_\ell)
\end{equation}
where for $i<\ell$,
\begin{eqnarray*}
&F_i(x_1,\ldots, x_i)=\int F(x_1,x_2, \ldots, x_\ell)\ d\mu (x_{i+1})\cdots 
d\mu(x_\ell)\\
&\quad -\int F(x_1,x_2, \ldots, x_\ell) \,d\mu (x_i)\cdots d\mu(x_\ell)\nonumber
\end{eqnarray*}
and
\[
F_\ell(x_1,x_2, \ldots, x_\ell)=F(x_1,x_2, \ldots, x_\ell) -\int F(x_1,x_2, 
\ldots, x_\ell)\, d\mu(x_\ell)
\]
which ensures, in particular, that
\begin{equation}\label{2.9}
\int F_i(x_1, x_2,\ldots,x_{i-1}, x_i)\,d\mu(x_i)\equiv 0 \quad\forall 
\quad x_1, x_2,\ldots, x_{i-1}.
\end{equation}
These enable us to write
\begin{equation}\label{2.10}
\Xi(t)=\sum_{i=1}^\ell \Psi_i(it)
\end{equation}
where for $1\leq i\leq\ell $,
\begin{equation}\label{2.11}
\Psi_i(t)=\sum_{1\leq n\leq t/i} F_i(X(n), X(2n),\ldots, X(in)).
\end{equation}

As in \cite{Ph} we will write $Z(t)\ll a(t)$ a.s. for a family of random
variables $Z(t),t\geq 0$ and a positive function $a(t),t\geq 0$ if
$\lim\sup_{t\to\infty}|Z(t)/a(t)|<\infty$ almost surely (a.s.)

\begin{theorem}\label{thm2.2}
Suppose that Assumption \ref{ass2.1} holds true. Then without changing its 
distribution the process $\Xi(t),\, t\geq 0$ given by (\ref{1.1}) can be 
redefined on a richer probability space where there exists an 
$\ell$-dimensional Gaussian process $G(t)=(G_1(t),...,G_\ell(t))$ with 
stationary independent increments having covariances $EG_i(s)G_j(t)=
D_{ij}s\wedge t$, for some nonnegatively definite matrix $D=(D_{ij},\,
1\leq i,j\leq\ell)$ and such that for some constant $\gam>0$,
\begin{equation}\label{2.12}
\Xi(t)-Q(t)\ll t^{\frac 12-\gam}\quad\mbox{a.s.}
\end{equation}
where $Q(t)=\sum_{j=1}^\ell G_j(jt)$ is a Gaussian process having, in general,
dependent increments (see \cite{KV}).
\end{theorem}

As in \cite{Ki2} we will rely on the representation (\ref{2.10})
 but in \cite{Ki2} we were able to obtain strong
invariance principles only for each $\Psi_i$ separately while Theorem 
\ref{thm2.2} provides a strong invariance principle for the original 
process $\Xi$. The proof of Theorem \ref{thm2.2} will rely on strong
approximation of the vector process $\Psi(t)=(\Psi_1(t),...,\Psi_\ell(t))$
by the vector Gaussian process $G(t)$ which unlike \cite{Ki2} cannot be
done via the Skorokhod embedding and we will have to verify conditions of
another approach from \cite{Zh} (see also \cite{Ph} and references there).

Let
\begin{equation}\label{2.13}
R(s,t)=EQ(s)Q(t)=\sum_{1\leq i,j\leq\ell}D_{ij}((is)\wedge (jt))
\end{equation}
be the covariance function of the Gaussian process $Q(t),\, t\geq 0$. Observe
 that $R(rs,rt)=rR(s,t)$ for any $r>0$. This together with (\ref{2.12}) and
 (\ref{2.13}) enable us to rely on the law of iterated logarithm for Gaussian
 processes from \cite{Oo} which is formulated in terms of the reproducing
 kernel Hilbert space $H$ corresponding to the kernel $R(s,t)$. Namely, we
 obtain
 \begin{corollary}\label{cor2.3} Let $K=\{ h\in H:\,\| h\|_HR(1,1)\leq 1\}$
 where $\|\cdot\|_H$ denotes the norm in $H$. Then the sequence of random
 functions 
 \begin{equation}\label{2.14}
 f_n(t)=\Xi(nt)(2R(n,n)\ln\ln n)^{-1/2},\, n\geq 3
 \end{equation}
 is a.s. equicontinuous and its set of limit points is a.s. contained in $K$.
 \end{corollary}
 
 In order to conclude that the set of limit points of the sequence $f_n$
 coincides with $K$ we have to impose at least some nondegeneracy conditions
 on the pair of $F$ and the process $X(n),\, n\geq 0$ but this question is not
 quite clear yet. It seems natural to expect that for generic (in some sense)
 pairs of $F$ and $X(n),\, n\geq 0$ the set of limit points of the sequence
 $f_n$ will coincide with $K$.
 
 Next, we describe our results concerning an a.s. central limit theorem. For
 each $t\in [0,1]$ and $n\in\bbN$ set
 \begin{equation}\label{2.15}
 Q_n(t)=n^{-1/2}\Xi([nt])(1+[nt]-nt)+n^{-1/2}\Xi([nt]+1)(nt-[nt])
 \end{equation}
 which produces a random element of the space $\bbC[0,1]$ of continuous
  functions on $[0,1]$ considered with the supremum norm topology. 
  Denote also by $\eta_Q$ the probability distribution
  of the process $Q(t),\, t\in[0,1]$ on $\bbC[0,1]$ while, as usual, by 
  $\del_x$ with $x\in\bbC[0,1]$ we denote the unit mass at $x$.
  
  \begin{theorem}\label{thm2.4}
  Under Assumption 2.1 and notations above
  \begin{equation}\label{2.16}
  \lim_{n\to\infty}\frac 1{\ln n}\sum_{k=1}^nk^{-1}\del_{Q_k}=\eta_Q\quad
  \mbox{a.s.}
  \end{equation}
  where the limit is taken in the sense of weak convergence of measures
  on $\bbC[0,1]$.
  \end{theorem}
  
  This result can be rephrased saying that outside of a single probability
  zero set
  \begin{equation}\label{2.17}
  \lim_{n\to\infty}\frac 1{\ln n}\sum_{k=1}^nk^{-1}\del_{\phi(Q_k)}=
  \phi(\eta_Q)\quad\mbox{a.s.}
  \end{equation}
  for all measurable functions $\phi$ on $\bbC[0,1]$ which are continuous 
  $\eta_Q-$a.s., where $\phi(\eta_Q)$ is the image measure of $\eta_Q$ under
  $\phi$. Letting, in particular, $\phi_1(x)=x(1)$ we obtain that a.s. as
  $n\to\infty$,
  \begin{equation}\label{2.18}
  \lim_{n\to\infty}\frac 1{\ln n}\sum_{k=1}^nk^{-1}\del_{\Xi(k)k^{-1/2}}=
  \mbox{distribution of}\,\, Q(1).
  \end{equation}
  
  Our proof of Theorem \ref{thm2.4} will follow the scheme of \cite{Bro}
  (see also \cite{LP}) 
  deriving first an almost sure central limit theorem for the Gaussian
  process $Q$ from Theorem \ref{thm2.2} and then relying on the strong
  approximation result (\ref{2.12}) in place of the Skorokhod embedding
  (representation) employed in \cite{Bro} which does not work here.
  
  \begin{remark}\label{rem2.5}
  An analogy of Corollary 1.5 from \cite{Bro} concerning the arcsine
  distribution can also be obtained in our circumstances. Namely, let
  \[
  L_n=n^{-1}\#\{ k:\,\Xi(k)>0,\, 1\leq k\leq n\}
  \]
  and define $\phi:\,\bbC[0,1]\to\bbR$ by $\phi(x)=$Lebesgue measure of
  $\{ u\in[0,1]:\, x(u)>0\}$. It is easy to see that $\phi$ is 
  $\eta_Q$-a.s. continuous and it follows from Theorem \ref{thm2.4} that
  with probability one
  \begin{equation}\label{2.19}
  \lim_{n\to\infty}\frac 1{\ln n}\sum_{k=1}^nk^{-1}\del_{\phi(Q_k)}
  =\phi(\eta_Q)
  \end{equation}
  where the right hand side of (\ref{2.19}) can be interpreted as an analogy
  of the arcsine law corresponding to the Gaussian process $Q$ in the same 
  sense as the usual arcsine law corresponds to the standard Brownian motion.
  Now, with probability one
  \[
  \lim_{n\to\infty}|\phi(Q_n)-L_n|=0
  \]
  and employing an analogy of Lemma 2.12 from \cite{Bro} (see Lemma 
  \ref{lem4.2} below) we obtain that with probability one
   \begin{equation}\label{2.20}
  \lim_{n\to\infty}\frac 1{\ln n}\sum_{k=1}^nk^{-1}\del_{L_k}=\phi(\eta_Q).
  \end{equation}
  \end{remark}
  
  \begin{remark}\label{rem2.6}
  In addition to almost sure central limit theorem large deviations from the
   limit were studied, as well (see \cite{MS} and \cite{Hec}). These results
   were also based on strong approximations but it was crucial there to rely
   for that on the Skorokhod embedding into the Brownian motion where rather
   specific estimates are available. Namely, for these large deviations
   estimates the result of Theorem \ref{thm2.2} in the form (\ref{2.12})
   does not suffice since now we have to show that for each $\ve>0$,
   \begin{equation}\label{2.21}
   \lim_{t\to\infty}\frac 1{\ln t}\ln P\{\sup_{0\leq s\leq t}|\Xi(s)-Q(s)|
   >\ve\sqrt t\}=-\infty.
   \end{equation}
   We can do this for each $|\Psi_i(s)-G_i(s)|$ separately obtaining
   $G_i$ via the Skorokhod embedding as in \cite{Ki2} but for $\Xi$ we need
   a multidimensional version of strong approximations for which estimates
   of the form (\ref{2.21}) do not seem to be directly available though
   they should be possible under appropriate conditions. For instance,
   (\ref{2.21}) would hold true if for any $m\in\bbN,\, t>0$ and some $C,\ve>0$,
   \begin{equation}\label{2.22}
   E\sup_{0\leq s\leq t}|\Xi(s)-Q(s)|^m\leq Ct^{m(\frac 12-\ve)}.
   \end{equation}
   The latter inequality can be obtained via estimates of the form
   \[
   P\{ |A_n-B_n|\geq\al_n\}\leq\al_n,
   \]
   where $A_n=\Psi_i(n)-\Psi_i(n-1)$ and $B_n=G_i(n)-G_i(n-1)$, which appear
   in proofs of usual multidimensional strong approximation theorems (cf., 
   for instance, Section 2.4 in \cite{MP}).
   \end{remark}

  \section{Strong approximations}\label{sec3}\setcounter{equation}{0}

The following result which is a part of Corollary 3.6 from \cite{KV}
will be a basis of estimates here.
\begin{proposition}\label{prop3.1}
Let $\cG$ and $\cH$ be $\sig$-subalgebras on a probability space
$(\Om,\cF,P)$, $X$ and $Y$ be $d$-dimensional random vectors and $f=
f(x,\om)$ be collections of random variables that are continuously
 (or separable) dependent on $x\in\bbR^d$ for almost all $\omega$,
 measurable with respect to $\cH$ and satisfy
\begin{eqnarray*}
&\|f( x,\omega)-f( y,\omega)\|_{q}\le C_1 (1+|x|^\iota + |y|^\iota)
|x-y|^\ka\\
&\mbox{and}\,\,\|f(x,\omega)\|_{q}\le C_2 (1+|x|^\iota).\nonumber
\end{eqnarray*}
Set $\tilde f(x,\om)=E(f(x,\cdot)|\cG)(\om)$ and $g(x)=Ef(x,\om)$. 

(i) Assume that  $q\ge p$,   $1\ge \ka>\te>\frac dp$ and $\frac {1}{a}\geq
\frac {1}{p}+\frac {\iota+2}{m}+\frac {\del}{q}$. Then  
\begin{equation}\label{3.1}
\| E\big(f(X,\cdot)|\cG\big)-g(X)\|_a\leq c\big(1+\| X\|^{\io+2}_{m})
(\vp_{q,p}(\cG,\cH)
+\| X-E(X|\cG)\|^\delta_{q}\big)
\end{equation}
where $c=c(\io,\ka,\te,p,q,a,\del,d)>0$ depends only on the parameters in 
brackets. 

(ii) Furthermore, let
$x=(v,z)$ and $X=(\Pi,\Up)$, where $\Pi$ and $\Up$ are $d_1$ and 
$d-d_1$-dimensional
random vectors, respectively, and let $f(x,\om)=f(v,z,\om)$ satisfy the
conditions above in $x=(v,z)$. Set $\tilde g(v)=E[f(v,\Up(\om),\om)]$. 
Then
\begin{eqnarray}\label{3.2}
&\| E\big(f(\Pi,\Up,\cdot)|\cG\big)-\tilde g(\Pi)\|_a\leq c\,(1+\| X\|^{\io+2}_{m})\\
&\times\big(\vp_{q,p}(\cG,\cH)+\| \Pi-E(\Pi|\cG)\|^\delta_q+\| \Up-E(\Up|\cH_i)
\|^\delta_q\big).
\nonumber
\end{eqnarray}

(iii) Furthermore, for $a,p,q,\iota,m,\del$ satisfying conditions of (i),
\begin{eqnarray}\label{3.3}
&\| \tilde f(X(\om),\om)-\tilde f(Y(\om),\om)-g(X)+g(Y)\|_a\\
&\leq c\,\vp_{q,p}(\cG,\cH)\big(1+\| X\|^{\io+2}_{m}+\| 
Y\|^{\io+2}_{m}\big)\| X-Y\|^\del_{q}\nonumber
\end{eqnarray}
where $c=c(\io,\ka,\te,p,q,a,\del,d)>0$ depends only on the parameters in 
brackets.
\end{proposition}

Observe that the conditions of Proposition \ref{prop3.1} are satisfied
in all our applications below in view of (\ref{2.5}) and (\ref{2.6}).
We will rely on the following result which in a close form appears as 
Theorem 1.3 in \cite{Zh} (see also \cite{Ph} and references there).
\begin{theorem}\label{thm3.2}
Let $\{ M(n),\,\cG_n\}_{n=1}^\infty,\, M(n)=(M_1(n),...,M_\ell(n))$ be
a square-integrable sequence of $\bbR^\ell$-valued martingale differences on 
a probability space $(\Om,\cG,P)$. Define conditional covariance matrices
$\sig(n)=\sig_{ij}(n))_{1\leq i,j\leq\ell}$ by
\begin{equation*}
\sig_{ij}(n)=E(M_i(n)M_j(n)|\cG_{n-1})
\end{equation*}
and set $\Sig(n)=\sum_{k=1}^n\sig(k)$. Suppose that there exists a covariance
matrix $D$ and a constant $\gam\in(0,1)$ such that
\begin{equation}\label{3.4}
\Sig(n)-nD\ll n^{1-\gam}\,\,\mbox{a.s. or}\,\, E\|\Sig(n)-nD\|=O(n^{1-\gam}),
\end{equation}
where $\|\cdot\|$ is the Euclidean matrix norm, and
\begin{equation}\label{3.5}
\sum_{n=1}^\infty n^{\gam-1}E\big(\| M(n)\|^2\bbI_{\| M(n)\|^2\geq 
n^{1-\gam}}\big)<\infty.
\end{equation}
Then without changing its distribution the sequence $\{ M(n),\, n\geq 1\}$
can be redefined on a richer probability space where there exists a sequence
$\{ Y(n),\, n\geq 1\}$ of independent identically distributed (i.i.d.) random
vectors with the covariance matrix $D$ and an $\ell$-dimensional Brownian
motion $Z(t),\, t\geq 0$ with the covariance matrix $D$ such that
\begin{equation}\label{3.6}
\sum_{k\leq n}(M(k)-Y(k))\ll n^{\frac 12-\ka}\quad\mbox{and}
\end{equation}
\begin{equation}\label{3.7}
\sum_{k\leq t}M(k)-Z(t)\ll t^{\frac 12-\ka}
\end{equation}
where $\ka>0$ does not depend on $n$ or on $t$.
\end{theorem}

Actually, only (\ref{3.6}) is obtained in \cite{Zh} but, in fact, (\ref{3.7})
follows from (\ref{3.6}) by a standard argument based on the Kolmogorov
extension theorem (see Section 2.4 in \cite{MP}). We observe also that if
the original probability space $(\Om,\cF,P)$ is already rich enough to have 
a uniformly distributed random variable independent of the sequence 
$\{ M(n),\, n\geq 1\}$ then the sequence $\{ Y(n),\, n\geq 1\}$ and the 
process $\{ Z(t),\, t\geq 0\}$ can already be constructed on $(\Om,\cF,P)$
and there is no need in enrichment and in redefinitions.

Our Theorem \ref{thm2.2} will follow immediately from the following result.

\begin{proposition}\label{prop3.3} Suppose that Assumption \ref{ass2.1} holds
 true. Then without changing its distributions the vector process 
 $\Psi(t)=(\Psi_1(t),...,\Psi_\ell(t)),\, t\geq 0$ (with $\Psi_i$ the same as 
 in (\ref{2.9})) can be redefined on a richer probability space
 where there exists an $\ell$-dimensional Gaussian process $G(t)=(G_1(t),...,
 G_\ell(t))$ with stationary independent increments having covariances
 $EG_i(s)G_j(t)=a_{ij}s\wedge t$ for some nonnegatively definite matrix
 $A=(a_{ij},\, 1\leq i,j\leq\ell)$ and such that for some constant $\gam>0$,
 \begin{equation}\label{3.8}
 \Psi(t)-G(t)\ll t^{\frac 12-\gam}\quad\mbox{a.s.}
 \end{equation}
 \end{proposition}
 \begin{proof}
 First, we recall the block construction from \cite{Ki2}.
 We will use the following notations from \cite{Ki2}
\begin{eqnarray}\label{3.9}
&F_{i,r,n}(x_1,x_2,\ldots, x_{i-1},\omega)
=E\big(F(x_1,x_2,\ldots, x_{i-1},X(n))|\cF_{n-r,n+r}\big),\\
&X_r(n)=E\big(X(n)|\cF_{n-r,n+r}\big),\,\,Y_i(q_i(n))=F_i(X(q_1(n)),\ldots,
 X(q_i(n)))\quad\mbox{and}\nonumber\\
&Y_i(j)=0\quad{\rm if}\quad j\ne q_i(n)\quad\mbox{for any}\,\, n,\,\,
Y_{i,r}( q_i(n))=F_{i,r,q_i(n)} (X_r(q_1(n)),\nonumber\\
&\ldots, X_r(q_{i-1}(n)),\omega)\quad\mbox{and}\quad Y_{i,r}(j)=0 
\quad{if}\quad j\ne q_i(n)\quad\mbox{for any}\,\, n.\nonumber
\end{eqnarray}
Next, we fix some positive numbers $4\eta<2\te<\tau<\del/4$ where
$\del$ is the same as in Assumption \ref{ass2.1}. 
 Now, we introduce pairs of "big" and "small" increasing
 blocks (somewhat differently than in \cite{Ki2}) defining for each $i$
  random variables $V_i(j)$ and $W_i(j)$ inductively so that
\begin{eqnarray}\label{3.10}
&\quad V_i(1)=Y_{i,1}(q_i(1)),\, W_i(1)=Y_{i,1}(q_i(2)),\, a(1)=0,\, b(1)=1\,\,
\mbox{and for}\,\, j>1,\\
& a(j)=b(j-1)+[(j-1)^\te],\,\, b(j)=a(j)+[j^\tau],\,\, r(j)=[j^\eta],
\nonumber\\
&V_i(j)=\sum_{a(j)< il\leq b(j)}Y_{i,r(j)}(il)\,\,
\mbox{and}\,\, W_i(j)=\sum_{b(j)<il\leq a(j+1)}Y_{i,r(j)}(il).\nonumber
\end{eqnarray}
 Next, set
\begin{equation}\label{3.11}
R_i(m)=\sum_{j=m+1}^\infty E\big( V_i(j)|\cG_m\big)
\end{equation}
and $M_i(m)=V_i(m)+R_i(m)-R_i(m-1)$ where $\cG_m=\cF_{-\infty, b(m)+r(m)}$. 
In the same way as in Section 3 of \cite{Ki2} we see that 
$(M_i(m),\,\cG_m)_{m\geq 1}$ is a martingale differences sequence for each
$i=1,...,\ell$, and so $M(m)=(M_1(m),...,M_\ell(m)),\, m\geq 1$ is a vector
martingale differences sequence. Now, proceeding similarly to Section 3 in 
\cite{Ki2} we obtain that for some $\ve>0$,
\begin{equation}\label{3.12}
\|\Psi(t)-\sum_{1\leq j\leq\nu(t)}M(j)\|\ll t^{\frac 12-\ve}\quad\mbox{a.s}
\end{equation}
where $\nu(t)=\max\{ j:\, b(j)+[j^\te]\leq t\}$ and $\|\cdot\|$ is the
 Euclidean norm in $\bbR^\ell$.
 
 In order to complete the proof of Proposition \ref{prop3.3} it remains 
 to veryfy the conditions of Theorem \ref{thm3.2} for the vector martingale
 differences sequence  $(M(m),\,\cG_m)_{m\geq 1}$. In Section 4 of \cite{Ki2}
 we showed relying on Proposition \ref{prop3.1} that for each $i=1,...,\ell$,
 \begin{equation}\label{3.13}
 E\Psi_i^2(t)-\sum_{1\leq l\leq\nu(t)} M_i^2(l)\ll t^{1-\ve}\quad\mbox{a.s}
 \end{equation}
 for some $\ve>0$ independent of $t$. Essentially, the same proof shows that
 for any $i,j=1,...,\ell$,
 \begin{equation}\label{3.14}
 E\Psi_i(t)\Psi_j(t)-\sum_{1\leq l\leq\nu(t)} M_i(l)M_j(l)\ll t^{1-\ve}
 \quad\mbox{a.s}
 \end{equation}
 for some $\ve>0$ independent of $t$. Existence and a description of the
 limit 
 \begin{equation}\label{3.15}
 \lim_{t\to\infty}t^{-1}E\Psi_i(t)\Psi_j(t)=D_{ij}
 \end{equation}
 was provided by Proposition 4.1 from \cite{KV}. In fact, it turns out that
 \begin{equation}\label{3.16}
 |E\Psi_i(t)\Psi_j(t)-D_{ij}t|\leq C_0
 \end{equation}
 for some $C_0>0$ independent of $t$ which is actually hidden inside of the
 proof in \cite{KV}, though, of course, a bound $C_0t^{1-\ve},\, \ve>0$
 in the right hand side of (\ref{3.16}) would suffice for our purposes, as well.
 The corresponding arguments are essentially contained at the end of Section 4
 in \cite{Ki2} but for readers' convenience we explain them here too.
 
 Recall, that according to Proposition 4.1 of \cite{KV} (see also Lemma 4.4
 there) the limit in (\ref{3.15}) can be written in the form
 \begin{equation}\label{3.17}
D_{ij}=\frac {\up}{ij}\sum_{u=-\infty}^\infty a_{ij}(u,2u,...,\up u)
\end{equation}
where $\up$ is the greatest common divisor of $i$ and $j$ with $i=\up i'$,
$j=\up j'$ and $i',j'$ being coprime and
\begin{eqnarray}\label{3.18}
&a_{ij}(u,2u,...,\up u)=\int F_i(x_1,...,x_i)F_j(y_1,...,y_j)\\
&\prod_{\sig\not\in(i',2i',...,\up i')}d\mu(x_\sig)
\prod_{\sig'\not\in(j',2j',...,\up j')}d\mu(y_{\sig'})\prod_{\eta=1}^\up
d\mu_{\eta u}(x_{\eta i'},y_{\eta j'})\nonumber
\end{eqnarray}
where $d\mu_0(x,y)=\del_{x,y}d\mu(x)$ is the measure supported by the
diagonal.

Next, we have
\begin{equation}\label{3.19}
E\Psi_i(t)\Psi_j(t)=\sum_{0\leq n\leq t/i,\, 0\leq n'\leq t/j}b_{ij}(n,n')
\end{equation}
where
\[
b_{ij}(n,n')=EF_i(X(n),...,X(in))F_j(X(n'),...,X(jn')).
\]
Suppose that $|in-jn'|=m\leq\frac 1{4\ell}\max(n,n')$ and let, for instance,
$in-jn'=m$. Then $\ell n\geq n'\geq (n-m)/\ell$, and so $\max(n,n')\leq\ell n$
and $m\leq n/4$. It follows that 
\begin{eqnarray*}
 &\min(in,jn')-\max((i-1)n,(j-1)n')=\min(n-m,n')\\
 &\geq (n-m)/\ell\geq\frac {3n}{4\ell}\geq\frac {3\max(n,n')}{4\ell^2}.
\end{eqnarray*}
The same argument works for the case $jn'-in=m$, as well.
Hence we can apply Proposition \ref{prop3.1}(ii) with 
\begin{eqnarray*}
&\cG=\cF_{-\infty,\min(in,jn')-\frac 1{2\ell^2}\max(n,n')},\,\cH=
\cF_{\min(in,jn')-\frac 1{4\ell^2}\max(n,n')},\\
& \Pi=(X(n),...,X((i-1)n);X(n'),...,X((j-1)n')))\,\,\mbox{and}
\,\,\Up=(X(in),X(jn'))
\end{eqnarray*}
which yields that
\begin{eqnarray}\label{3.20}
&|b_{ij}(n,n')-\int EF_i(X(n),...,X((i-1)n),x)F_j(X(n'),...,\\
&X((i-1)n'),y)d\mu(x,y)|\leq C_1(\vp_{q,p}(\frac 1{4\ell^2}\max(n,n'))
+\be_q^\del(\frac 1{4\ell^2}\max(n,n')))\nonumber
\end{eqnarray}
for some $C_1>0$ independent of $n$ and $n'$.
 
We proceed by induction dealing with the case $|in-jn'|=m>\frac 1{4\ell}
\max(n,n')$ by the argument below. Suppose that we already proved that
for some $k<i$ and $k'<j$,
\begin{eqnarray}\label{3.21}
&|b_{ij}(n,n')-\int EF_i(X(n),...,X(kn),x_{k+1},...,x_i)F_j(X(n'),...,\\
&X(k'n'),y_{k'+1},...,y_j)d\nu(x_{k+1},...,x_i,y_{k'+1},...,y_j)|\leq\ve_{k,k'}
(n,n').\nonumber
\end{eqnarray}
If $|kn-k'n'|\leq\frac 1{4\ell^2}\max(n,n')$ then in the same way as in 
(\ref{3.20}) we obtain that 
\begin{eqnarray}\label{3.22}
&|b_{ij}(n,n')-\int EF_i(X(n),...,X((k-1)n),x_k,...,x_i)F_j(X(n'),...,\\
&X((k'-1)n'),y_{k'},...,y_j)d\mu_{kn-k'n'}(x_k,y_{k'})\nu(x_{k+1},...,x_i,
y_{k'+1},...,y_j)|\nonumber\\
&\leq\ve_{k,k'}(n,n')+C_2(\vp_{q,p}(\frac 1{4\ell^2}\max(n,n'))
+\be_q^\del(\frac 1{4\ell^2}\max(n,n')))\nonumber
\end{eqnarray}
for some $C_2>0$ independent of $n$ and $n'$. On the other hand, if
$|kn-k'n'|>\frac 1{4\ell}\max(n,n')$ then
\[
\max(kn,k'n')\geq\max\big(\min(kn,k'n'),\,\max((k-1)n,(k'-1)n')\big)+
\frac 1{4\ell^2}\max(n,n'),
\]
and so we can apply again Proposition \ref{prop3.1}(ii) with
\begin{eqnarray*}
&\cG=\cF_{-\infty,\max(kn,k'n')-\frac 1{6\ell^2}\max(n,n')},\,\cH=
\cF_{\max(kn,k'n')-\frac 1{12\ell^2}\max(n,n')},\\
& \Pi=\big(X(n),...,X((k-1)n);X(n'),...,X((k'-1)n');X(\min(kn,k'n'))\big)
\end{eqnarray*}
and $\Up=X(\max(kn,k'n'))$ to obtain that
\begin{eqnarray}\label{3.23}
&|b_{ij}(n,n')-\int EF_i(X(n),...,X((k-1)n),U_{kn},x_{k+1},...,x_i)F_j(X(n'),
...,\\
&X((k'-1)n'),U_{k'n'},y_{k'+1},...,y_j)d\mu(U_{\max(kn,k'n')})\nu(x_{k+1},...
,x_i,y_{k'+1},...,y_j)|\nonumber\\
&\leq\ve_{k,k'}(n,n')+C_3(\vp_{q,p}(\frac 1{12\ell^2}\max(n,n'))
+\be_q^\del(\frac 1{12\ell^2}\max(n,n')))\nonumber
\end{eqnarray}
for some $C_3>0$ independent of $n$ and $n'$ where $U_{\min(kn,k'n')}=
X(\min(kn,k'n'))$. In particular, if $k=i$ and $k'=j$ we obtain from
(\ref{2.9}) that
\begin{equation}\label{3.24}
|b_{ij}(n,n')|\leq C_3(\vp_{q,p}(\frac 1{12\ell^2}\max(n,n'))
+\be_q^\del(\frac 1{12\ell^2}\max(n,n'))).
\end{equation}
This together with the above induction argument yields that
\begin{equation}\label{3.25}
|b_{ij}(n,n')-a_{ij}(u,2u,...,\up u)|\leq C_4(\vp_{q,p}(\frac 1{12\ell^2}
\max(n,n'))+\be_q^\del(\frac 1{12\ell^2}\max(n,n')))
\end{equation}
for some $C_4>0$ independent of $n$ and $n'$ provided $ni-jn'=\up u$ and
$\up$ is the greatest common divisor of $i$ and $j$ with $a_{ij}$ defined
by (\ref{3.18}).

It is not difficult to see (the explanation can be found in the proof of
Lemma 4.4 of \cite{KV}) that the number of integer solutions of $in-jn'
=\up u$ with $in,\, jn'\leq t$ can differ from $[\up t(ij)^{-1}]$ by at
most a constant independent of $t$. This together with 
(\ref{3.17})--(\ref{3.19}), (\ref{3.25}) and Assumption \ref{ass2.1} 
yields (\ref{3.16}). Taking into account (\ref{3.14}) we conclude that 
the condition (\ref{3.4}) of Theorem \ref{thm3.2} is satisfied for the
vector martingale differences sequence $\{ M(j),\, 1\leq j\leq\nu(t)\}$
constructed above.

Next, we verify the condition (\ref{3.5}). By the Cauchy--Schwarz and
the Chebyshev inequalities
\begin{eqnarray}\label{3.26}
&A(n)=\frac 1{n^{1-\gam}}E\| M(n)\|^2\bbI_{\{\| M(n)\|^2\geq n^{1-\gam}\}}\\
&\leq \frac 1{n^{1-\gam}}\big(E\| M(n)\|^4\big)^{1/2}\big( P\{\| M(n)\|^2\geq
 n^{1-\gam}\}\big)^{1/2}\leq\frac 1{n^{2(1-\gam)}}E\| M(n)\|^4.\nonumber
\end{eqnarray}
Now
\begin{eqnarray}\label{3.27}
&\| M(n)\|\leq\sum_{i=1}^\ell |M_i(n)|\leq \sum_{i=1}^\ell\big(|V_i(n)|+|R_i(n)|
-|R_i(n-1)|\big)\\
&\leq\sum_{i=1}^\ell\sum_{a(n)< il\leq b(n)}|Y_{i,r(n)}|+
\sum_{i=1}^\ell\big(|R_i(n)|-|R_i(n-1)|\big).\nonumber
\end{eqnarray}
Since $(a_1+\cdots +a_k)^4\leq k^3(a_1^4+\cdots a_k^4)$ we conclude from here
relying on (\ref{2.6}), Assumption \ref{ass2.1} and the H\" older inequality that
for any $n\geq 1$,
\begin{equation}\label{3.28}
E\| M(n)\|^2\leq C_5n^{4\tau}
\end{equation}
where $C_5>0$ is independent of $n$, and so 
\[
A(n)\leq C_5 n^{4\tau-2(1-\gam)}\leq C_5n^{\del-2(1-\gam)}.
\]
Since $\del<1$ we can choose $\gam>0$ so small that $\del-2(1-\gam)<-1$
whence $\sum_{n=1}^\infty A(n)<\infty$ and the condition (\ref{3.5}) holds
true completing the proof of Proposition \ref{prop3.3}.
 \end{proof}
 
 Finally, (\ref{2.12}) follows from (\ref{3.8}) in view of (\ref{2.10}) 
 concluding the proof of Theorem \ref{thm2.2}.    \qed

\section{Almost sure central limit theorem}\label{sec4}\setcounter{equation}{0}

We start the proof of Theorem \ref{thm2.4} with the following
\begin{proposition}\label{prop4.1}
Set
\begin{equation}\label{4.1}
Q^{(s)}_u=s^{-1/2}Q(us)=s^{-1/2}\sum_{j=1}^\ell G_j(ujs),\quad u\in [0,1]
\end{equation}
and define random measures on $\bbC[0,1]$ by 
\begin{equation}\label{4.2}
\zeta_t(\om)=\frac 1{\ln t}\int_1^t\frac {ds}s\del_{Q^{(s)}_\bullet(\om)},\,\,
t>1
\end{equation}
where $\del_{Q^{(s)}_\bullet(\om)}$ is the unit mass concentrated on the 
element $Q^{(s)}_\bullet(\om)\in\bbC[0,1]$. Set also
\begin{equation}\label{4.3}
\nu_n(\om)=\frac 1{\ln n}\sum_{k=1}^nk^{-1}\del_{Q^{(k)}_\bullet},\,\, n\geq 2.
\end{equation}
Then with probability one
\begin{equation}\label{4.4}
\lim_{t\to\infty}\zeta_t=\eta_Q\,\,\mbox{and}\,\,\lim_{n\to\infty}\nu_n=\eta_Q
\end{equation}
where, again, the limit is taken in the sense of weak convergence of 
measures on $\bbC[0,1]$.
\end{proposition}
\begin{proof} In the same way as in \cite{Bro} we show first that with 
probability one the measures $\{\zeta_t,\, t\geq e\}$ are tight observing
that the estimates for the Brownian motion in \cite{Bro} go through for
our process $Q$, as well, since it is a linear combination of linearly
time changed Brownian motions.

Next, as in \cite{Bro} we define $\phi:\,\bbC[0,1]\to\bbR$ by
\begin{equation}\label{4.5}
\phi(x)=E^{i(\al_1x(u_1)+\cdots+\al_kx(u_k))}
\end{equation}
for some $k\in\bbN$, $0<u_1<\cdots<u_k\leq 1$ and $\al_1,...,\al_k\in\bbR$.
We want to show that 
\begin{equation}\label{4.6}
\lim_{t\to\infty}\int_{\bbC[0,1]}\phi(x)\zeta_t(\om)(dx)=\int_{\bbC[0,1]}
\phi(x)d\eta_Q(x).
\end{equation}
Set
\begin{equation}\label{4.7}
\Phi_s=\exp\big( i(\al_1\sqrt {u_1}R_{u_1s}+\cdots +\al_k\sqrt {u_k}R_{u_ks}
)\big)
\end{equation}
where $R_s=Q_1^{(s)}$. Then
\begin{equation}\label{4.8}
\int_{\bbC[0,1]}\phi(x)\zeta_t(\om)(dx)=\frac 1{\ln t}\int_1^t\frac {ds}s
\Phi_s
\end{equation}
since $\sqrt uR_{us}=\sqrt uQ_1^{(us)}=Q_u^{(s)}$. Next, observe that
the laws of $\{ R_s,\, s>0\}$ and $\{ R_{sh},\, s>0\}$ (as processes)
 coincide for each $h>0$ since both processes are Gaussian with zero
 mean and in view of Theorem \ref{thm2.2} the covariance function
 \begin{equation}\label{4.9}
 ER_{sh}R_{th}=h^{-1}EQ_h^{(s)}Q_h^{(t)}=h^{-1}\sum_{i,j=1}^\ell EG_i(his)
 G_j(hjs)=s\sum_{i,j=1}^\ell D_{ij}(i\wedge j)
 \end{equation}
 does not depend on $h$. It follows from here and the definition (\ref{4.7})
 that if we set $\hat\Phi_s=\Phi_{\ln s}$ then the laws of $\{\hat\Phi_s,\,
 s\in\bbR\}$ and $\{\hat\Phi_{s+h},\, s\in\bbR\}$, $h\in\bbR$ coincide. Then
 by the Birkhoff ergodic theorem together with the fact that the tail
 $\sig$-algebra $\cap_{t>0}\sig\{ G_1(u),G_2(u),...,G_\ell(u);\, u>t\}$
 is trivial (as for an $\ell$-dimensional Brownian motion) we conclude that
 with probability one 
 \begin{eqnarray}\label{4.10}
 &\lim_{n\to\infty}\frac 1n\int_1^{e^n}\frac {ds}s\Phi_s=\lim_{n\to\infty}
 \frac 1n\int_0^n\Phi_{e^u}du\\
 &=\lim_{n\to\infty}\frac 1n\int_0^n\hat\Phi_udu=E\hat\Phi_0=E\Phi_1.
 \nonumber\end{eqnarray}
 A simple comparison of integrals $\int_1^{e^t}$ and $\int_1^{e^{[t]}}$ as
 in \cite{Bro} shows also that
 \begin{equation}\label{4.11}
 \lim_{t\to\infty}\frac 1t\int_1^{e^t}\frac {ds}s\Phi_s=E\Phi_1.
 \end{equation}
 But by (\ref{4.5}) and (\ref{4.7}),
 \begin{equation}\label{4.12}
 E\Phi_1=E\exp(i(\al_1Q_{u_1}^{(1)}+\cdots +\al_kQ_{u_k}^{(1)}))=
 \int_{\bbC[0,1]}\phi(x)d\eta_Q(x)
 \end{equation}
 since $Q_u^{(1)}=Q(u)$. Hence, (\ref{4.6}) follows from (\ref{4.7}),
 (\ref{4.8}), (\ref{4.11}) and (\ref{4.12}). Since (\ref{4.6}) holds true
 for any $\phi$ defined by (\ref{4.5}) then relying on tightness of the
 family $\{\zeta_t,\, t\geq e\}$ we obtain the first limit in (\ref{4.4}).
 The second limit there holds true by the same arguments as in \cite{Bro}
 which are just estimates for the Brownian motion valid in our case of
 a linear combination of linearly time changed Brownian motions, as well.
 \end{proof}

In order to derive Theorem \ref{thm2.4} from Proposition \ref{prop3.3} we
 will rely on the following result which appears in \cite{Bro} as Lemma 2.12.
\begin{lemma}\label{lem4.2} Let $\Pi,\Up:\,\bbR^+\to\bbC[0,1]$ or
$\Pi,\Up:\,\bbN^+\to\bbC[0,1]$ be measurable $\bbC[0,1]$-valued stochastic
processes such that, respectively,
\begin{equation}\label{4.13}
\lim_{s\to\infty}\|\Pi_s-\Up_s\|_{\bbC[0,1]}=0\,\,\mbox{or}\,\,
\lim_{n\to\infty}\|\Pi_n-\Up_n\|_{\bbC[0,1]}=0
\end{equation}
where $\|\cdot\|_{\bbC[0,1]}$ is the supremum norm on $\bbC[0,1]$.
Then for all bounded, uniformly continuous functions $f:\,\bbC[0,1]\to\bbR$,
\begin{equation}\label{4.14}
\lim_{t\to\infty}\big(\frac 1{\ln t}\int_1^t\frac {ds}s(f(\Pi_s)-f(\Up_s))\big)
=0\,\,\mbox{or}\,\,
\lim_{n\to\infty}\big(\frac 1{\ln n}\sum_{k=1}^nk^{-1}(f(\Pi_k)-f(\Up_k))\big)
=0,
\end{equation}
respectively.
\end{lemma}

In order to prove Theorem \ref{thm2.4} it suffices in view of Proposition
\ref{prop4.1} and Lemma \ref{lem4.2} to show that
\begin{equation}\label{4.15}
 \lim_{n\to\infty}\| Q^{(n)}-Q_n\|_{\bbC[0,1]}=0\quad\mbox{a.s.}
 \end{equation}
 where $Q^{(n)}$ and $Q_n$ are defined by (\ref{4.1}) and (\ref{2.15}),
 respectively.
 For each $t\in[0,1],\, n\in\bbN$ define $\Psi^{(n)}(t)=(\Psi^{(n)}_1,...,
 \Psi^{(n)}_\ell)$ and $G^{(n)}(t)=(G^{(n)}_1,...,G^{(n)}_\ell)$ where
 for $j=1,...,\ell$,
 \begin{equation}\label{4.16}
 \Psi^{(n)}_j(t)=n^{-1/2}\Psi_j(j[nt])(1+[nt]-nt)+n^{-1/2}\Psi_j(j[nt]+j)
 (nt-[nt]),
 \end{equation}
 \begin{equation}\label{4.17}
 G_j^{(n)}(t)=n^{-1/2}G_j(j[nt])(1+[nt]-nt)+n^{-1/2}G_j(j[nt]+j)
 (nt-[nt])
 \end{equation}
 and $H_j^{(s)}(t)=s^{-1/2}G_j(tjs)$ where $\Psi_j$ and $G_j$ are the same
 as in (\ref{2.11}) and Theorem \ref{thm2.2}, respectively. Then in order to
 obtain (\ref{4.15}) it suffices to show that for each $j=1,...,\ell$,
 \begin{equation}\label{4.18}
 \lim_{n\to\infty}\| H^{(n)}_j-G^{(n)}_j\|_{\bbC[0,1]}=0\quad\mbox{a.s.}\quad
 \mbox{and}
 \end{equation}
 \begin{equation}\label{4.19}
 \lim_{n\to\infty}\| G^{(n)}_j-\Psi^{(n)}_j\|_{\bbC[0,1]}=0\quad\mbox{a.s.}
 \end{equation}
 
 Now, in the same way as in \cite{Bro} for each $\ve>0$ by the Doob martingale
 inequality,
 \begin{eqnarray}\label{4.20}
 &P\{\|H^{(n)}_j-G^{(n)}_j\|_{\bbC[0,1]}\geq \frac 12\ve\sqrt n\}\\
 &\leq n\sum_{k=0}^{n-1}P\{ \sup_{k\leq t\leq k+1}|G_j(jt)-G_j(jk)|\geq
 \frac 12\ve\sqrt n\}\leq C\ve^{-6}n^{-2}.\nonumber
 \end{eqnarray}
 This together with the Borel--Cantelli lemma yields (\ref{4.19}). Finally,
 \begin{equation}\label{4.21}
 \| G^{(n)}_j-\Psi^{(n)}_j\|_{\bbC[0,1]}\leq n^{-1/2}\max_{1\leq k\leq n+1}
 |G_j(jk)-\Psi_j(jk)|
 \end{equation}
 and the right hand side of (\ref{4.21}) converges with probability one
 to 0 as $n\to\infty$ in view of Theorem \ref{thm2.2}, completing the
 proof of Theorem \ref{thm2.4}.  \qed
 
 \begin{remark}\label{rem4.3} A slightly different method of proof of the
 almost sure central limit theorem from \cite{LP} can also be adapted to
 our nonconventional setup.
 \end{remark}
\bibliography{matz_nonarticles,matz_articles}
\bibliographystyle{alpha}

\end{document}